\documentclass[10pt]{article}

 \usepackage{amsmath}
\usepackage{amsthm,amssymb}
 \usepackage{makeidx,epsfig,lscape}
 \usepackage{color,colortbl}
 \usepackage{fancyhdr}
 \usepackage{times}
 \usepackage{xcolor,pict2e}
 \usepackage[latin1]{inputenc}

\newcommand{\N}{\mathbb N}

\newcommand{\gt}{\gamma_{t}}

 \renewcommand{\headrulewidth}{0pt}
 \renewcommand{\footrulewidth}{0.5pt}

 \definecolor{myaqua}{rgb}{0.0,0.5,0.55}
 \definecolor{lightaqua}{rgb}{0.75,0.95,0.95}

 \usepackage[colorlinks = true,
            linkcolor = myaqua,
            urlcolor  = blue,
            citecolor = red,
            pdfpagemode=UseOutlines,
            bookmarksnumbered=true,
            pdfpagelabels,
            breaklinks]{hyperref}

\usepackage{caption}
\usepackage{floatrow}

\newtheorem{theorem}{Theorem}

\newtheorem{lem}{Lemma}
\newtheorem{coro}{Corollary}
\newtheorem{defn}{Definition}[section]
\newtheorem{rem}{Remark}[section]
\def\lin#1#2{\textcolor[rgb]{0.6,0.6,0.6}{\vspace*{#1mm} \hrule
   height 3 pt \vspace*{#2mm}}}
\def\bt{\begin{tabular}}
\def\et{\end{tabular}}
\def\and{\mbox{ and }}

\def\P{\mbox{\bf P}}

\def\1{{\bf 1}}

 \def\boxx#1#2#3#4#5{
 {\linethickness{#4pt}\put(#1,#5){\color{myaqua}{\line(1,0){#3}}}}
 \multiput(#1,#2)(0,#4){2}{\line(1,0){#3}}
 \multiput(#1,#2)(#3,0){2}{\line(0,1){#4}}
  }

\begin{document}


 $\mbox{ }$

 \vskip 12mm

{ 

{\noindent{\Large\bf\color{myaqua}
  Characterization of a new class of stochastic processes including all known extensions of the class \texorpdfstring{$(\Sigma)$}{sigma}} 
%
\\[6mm]
{\bf Fulgence EYI OBIANG$^{1,a}$, Paule Joyce MBENANGOYE$^{1,b}$ and Octave MOUTSINGA$^{1,c}$}}
\\[2mm]
{ 
$^1$URMI Laboratory, Département de Mathématiques et Informatique, Faculté des Sciences, Université des Sciences et Techniques de Masuku, BP: 943 Franceville, Gabon. 
\\
$^a$ Email: \href{mailto:feyiobiang@yahoo.fr}{\color{blue}{\underline{\smash{feyiobiang@yahoo.fr}}}}\\[1mm]
$^b$ Email: \href{mailto:paulejoycembenangoye@yahoo.fr}{\color{blue}{\underline{\smash{paulejoycembenangoye@yahoo.fr}}}}\\[1mm]
$^c$ Email: \href{mailto:octavemoutsing-pro@yahoo.fr}{\color{blue}{\underline{\smash{octavemoutsing-pro@yahoo.fr}}}}\\[1mm]
\lin{5}{7}

 {  
 {\noindent{\large\bf\color{myaqua} Abstract}{\bf \\[3mm]
 \textup{
 This paper contributes to the study of class $(\Sigma^{r})$ as well as the càdlàg semi-martingales of class $(\Sigma)$, whose finite variational part is càdlàg instead of continuous. The two above-mentioned classes of stochastic processes are extensions of the family of càdlàg semi-martingales of class $(\Sigma)$ considered by Nikeghbali \cite{nik} and Cheridito et al. \cite{pat}; i.e., they are processes of the class $(\Sigma)$, whose finite variational part is continuous. The two main contributions of this paper are as follows. First, we present a new characterization result for the stochastic processes of class $(\Sigma^{r})$. More precisely, we extend a known characterization result that Nikeghbali established for the non-negative sub-martingales of class $(\Sigma)$, whose finite variational part is continuous (see Theorem 2.4 of \cite{nik}). Second, we provide a framework for unifying the studies of classes $(\Sigma)$ and $(\Sigma^{r})$. More precisely, we define and study a new larger class that we call class $(\Sigma^{g})$. In particular, we establish two characterization results for the stochastic processes of the said class. The first one characterizes all the elements of class $(\Sigma^{g})$. Hence, we derive two corollaries based on this result, which provides new ways to characterize classes $(\Sigma)$ and $(\Sigma^{r})$. The second characterization result is, at the same time, an extension of the above mentioned characterization result for class $(\Sigma^{r})$ and of a known characterization result of class $(\Sigma)$ (see Theorem 2 of \cite{fjo}). In addition, we explore and extend the general properties obtained for classes $(\Sigma)$ and $(\Sigma^{r})$ in \cite{nik,pat,mult, Akdim}. For instance, we study the positive and negative parts of the processes of class $(\Sigma^{g})$. We show that the product of the processes of class $(\Sigma^{g})$ with the vanishing quadratic covariation also belongs to class $(\Sigma^{g})$. Further, we show that every positive process $X$ of class $(\Sigma^{g})$ admits a multiplicative decomposition. In other words, it can be decomposed as
$$X=CW-1,$$
where $W$ is a positive local martingale with $W_{0}=1$, and $C$ is a non-decreasing process. This result is an extension of that obtained by Eyi Obiang et al. for positive non-negative càdlàg processes of class $(\Sigma)$ \cite{fjo}. We also present a result that enables the recovery of any process of class $(\Sigma^{g})$ from its final value $X_{\infty}$ and of an honest time $g$, which is the last time $(X_{t}:t\geq0)$ or $(X_{t-}:t\geq0)$ visited the origin. More precisely, this formula has the following form:
$$X_{t}=E\left[X_{\infty}1_{\{g\leq t\}}|\mathcal{F}_{t}\right],$$ 
where $X$ is the process of the class $(\Sigma^{g})$,$X_{\infty}=\lim_{t\to+\infty}{X_{t}}$, and $g=\sup\{t\geq0:X_{t}X{t-}=0\}$. 
 }}
 \\[4mm]
 {\noindent{\large\bf\color{myaqua} Keywords:}{\bf \\[3mm]
 class $(\Sigma)$; class $(\Sigma{r})$; Balayage formula; Honest time; Relative martingales.
}}\\[4mm]{\noindent{\large\bf\color{myaqua} MSC:}{\color{blue} 60G07; 60G20; 60G46; 60G48}}
\lin{3}{1}

\renewcommand{\headrulewidth}{0.5pt}
\renewcommand{\footrulewidth}{0pt}

 \pagestyle{fancy}
 \fancyfoot{}
 \fancyhead{} 
 \fancyhf{}
 \fancyhead[RO]{\leavevmode \put(-140,0){\color{myaqua} Fulgence EYI OBIANG et al. (2021)} \boxx{15}{-10}{10}{50}{15} }
 \fancyfoot[C]{\leavevmode
 \put(-2.5,-3){\color{myaqua}\thepage}}

 \renewcommand{\headrule}{\hbox to\headwidth{\color{myaqua}\leaders\hrule height \headrulewidth\hfill}}
\section*{Introduction}

This study investigates càdlàg semi-martingales of classes $(\Sigma)$ and $(\Sigma^{r})$. These are stochastic processes $X$ of the following form:
\begin{equation}\label{i1}
	X=M+A,
\end{equation}
where $M$ is a càdlàg local martingale with $M_{0}=0$ and $A$ is an adapted predictable process of finite variation with $A_{0}=0$, such that the signed measure induced by $A$ is carried by an optional random set $H$, where
\begin{equation}\label{i2}
	\int_{0}^{t}{1_{H^{c}}(s)dA_{s}}=0\text{,  }\forall t\geq0.
\end{equation}
Such processes are strongly related to many probabilistic studies. Well-known examples of studies where the use of such processes is capitalized include the theory of Azéma--Yor martingales, the study of zeros of continuous martingales \cite{1}, the study of Brownian local times, the balayage formulas in the progressive case \cite{mey}, the construction of solutions for skew Brownian motion equations \cite{fjo}, and the resolution of Skorokhod's reflection equation and embedding problem \cite{2}. These classes are represented in the form of $H$. More precisely, for the processes of class $(\Sigma)$, we have 
$$H=\{t\geq0:X_{t}=0\};$$
By contrast, for class $(\Sigma^{r})$, the random set $H$ takes the following form
$$H=\{t\geq0:X_{t-}=0\}.$$
The stochastic processes of class $(\Sigma)$, whose finite variational part is continuous, have been studied extensively by several authors,
 including Yor, Najnudel, Nikeghbali, Cheridito, Platen, Ouknine, Bouhadou, Sakrani, Eyi Obiang, Moutsinga, and Trutnau  (see \cite{siam,pat,eomt,fjo,naj,naj1,naj2,naj3,nik,mult,y1}). These authors studied the main properties of these processes, presented their applications, and relaxed the original hypotheses. The notion of stochastic processes of class $(\Sigma)$  has evolved over time, and the present study considers the most general definition presented by Eyi Obiang et al. in \cite{fjo}, which extends the notion of class $(\Sigma)$ to càdlàg semi-martingales, whose finite variational part is considered càdlàg instead of continuous. We consider the following definition:
\begin{defn}\label{di1}
We say that a semi-martingale $X$ is of class $(\Sigma)$ if it decomposes as $X=M+A$, where
\begin{enumerate}
	\item $M$ is a càdlàg local martingale, with $M_{0}=0$;
	\item $A$ is an adapted càdlàg predictable process with finite variations such that $A_{0-}=A_{0}=0$; 
	\item  $\int_{0}^{t}{1_{\{X_{s}\neq0\}}dA_{s}}=0$ for all $t\geq0$.
\end{enumerate}
\end{defn}

By contrast, the study of class $(\Sigma^{r})$ is quite recent. In 2018, Akdim et al. \cite{Akdim}  first characterized and studied the structural properties of the positive submartingales of the said class. However, it should be noted that the use of the processes of class $(\Sigma^{r})$ has a longer history. For instance, in 1981, Barlow \cite{bar} used these processes to show that any positive submartingale is equal to the absolute value of a martingale. More precisely, we consider the following definition:
\begin{defn}\label{di2}
We say that a semi-martingale $X$ is of class $(\Sigma)$ if it decomposes as $X=M+A$, where
\begin{enumerate}
	\item $M$ is a càdlàg local martingale, with $M_{0}=0$;
	\item $A$ is an adapted càdlàg predictable process with finite variations such that $A_{0-}=A_{0}=0$; 
	\item  $\int_{0}^{t}{1_{\{X_{s-}\neq0\}}dA_{s}}=0$ for all $t\geq0$.
\end{enumerate}
\end{defn}
Notably, the two above-mentioned classes coincide for the processes $X$,  whose finite variational part $A$ is considered continuous (i.e., class $(\Sigma)$ under the hypotheses considered by Nikeghbali \cite{nik} and Cheridito et al. \cite{pat}). However, it is possible to determine processes belonging to at least one of these classes that are not present in another class. 

 This study contributes toward existing literature by enriching the general framework and developing techniques for dealing with stochastic processes of class $(\Sigma^{r})$ and the càdlàg semi-martingales of class $(\Sigma)$, whose finite variational part is càdlàg instead of continuous. First, we study the processes of class $(\Sigma^{r})$ by proposing a new method to characterize such stochastic processes.

Second, we present a general framework that unifies the study of the two above-mentioned classes. More precisely, we propose a new larger class that includes all the processes of the classes $(\Sigma)$ and $(\Sigma^{r})$. We term this class as $(\Sigma^{g})$ and define it as follows:
\begin{defn}\label{d1}
We say that a stochastic process $X$ is of the class $(\Sigma^{g})$ if it decomposes as $X=M+A$, where
\begin{enumerate}
	\item $M$ is a càdlàg local martingale, with $M_{0}=0$;
	\item $A$ is an adapted càdlàg predictable process with finite variations such that $A_{0-}=A_{0}=0$; 
	\item  $\int_{0}^{t}{1_{\{X_{s}X_{s-}\neq0\}}dA_{s}}=0$ for all $t\geq0$.
\end{enumerate}
\end{defn}
Hence, we explore and extend the general properties obtained for classes $(\Sigma)$ and $(\Sigma^{r})$ in \cite{nik,pat,mult, Akdim}. For instance, we study the positive and negative parts of the processes of class $(\Sigma^{g})$ and show that the product of the processes of class $(\Sigma^{g})$ with vanishing quadratic covariation also belongs to class $(\Sigma^{g})$. Further, we show that every positive process $X$ of class $(\Sigma^{g})$ has a multiplicative decomposition. In other words, it can be decomposed as
$$X=CW-1,$$
where $W$ is a positive local martingale with $W_{0}=1$, and $C$ is a non-decreasing process. This result is an extension of that obtained by Nikeghbali for positive and continuous submartingales \cite{mult}. We also present a result that enables the recovery of any process of class $(\Sigma^{g})$ from its final value $X_{\infty}$ and of an honest time $g$, which is the last time $(X_{t}:t\geq0)$ or $(X_{t-}:t\geq0)$ visited the origin. More precisely, this formula has the following form:
$$X_{t}=E\left[X_{\infty}1_{\{g\leq t\}}|\mathcal{F}_{t}\right],$$ 
where $X$ is the process of class $(\Sigma^{g})$,$X_{\infty}=\lim_{t\to+\infty}{X_{t}}$, and $g=\sup\{t\geq0:X_{t}X{t-}=0\}$. Finally, we generalize the result of Nikeghbali (Theorem 2.1 of \cite{nik}), which affords a martingale characterization for positive processes of class $(\Sigma)$.

The  remainder of this paper is organized as follows. In Section \ref{s1}, we present some useful preliminaries and introduce new characterization of class $(\Sigma^{r})$. Section \ref{s2} is devoted to the study of the new class $(\Sigma^{g})$. Finally, Section \ref{s3}  summarizes the related approaches and methods.

\section{Preliminaries and new characterization of the class \texorpdfstring{$(\Sigma^{r})$}{sigma(r)}}\label{s1}
	
The main purpose of this section is to contribute toward the framework for studying the processes of class $(\Sigma^{r})$. More precisely, we propose a new method for characterizing the positive processes of class $(\Sigma^{r})$. However, we first recall some results and notations that will be useful for understanding this work.
\subsection{Notations and Preliminaries}

In this work, we fix a filtered probability space $(\Omega,(\mathcal{F}_{t})_{t\geq0},\mathcal{F}_{t},\P)$ that satisfies the usual conditions. Throughout this work, for any càdlàg stochastic process $X$, we consider that  $X^{c}$ is its continuous part, and $(X_{t-})_{t\geq0}$ denotes the process defined by $\forall t>0$, where $X_{t-}$ is the left limit of $X$ in $t$ and $X_{0-}=X_{0}$.

Now, let us recall the version of class $(\Sigma)$ studied by Nikeghbali \cite{nik} and Cheridito et al. \cite{pat}
\begin{defn}\label{dnik}
We say that a semi-martingale $X$ is of class $(\Sigma)$ if it decomposes as $X=M+A$, where
\begin{enumerate}
	\item $M$ is a càdlàg local martingale, with $M_{0}=0$;
	\item $A$ is an adapted continuous process with finite variations such that $A_{0}=0$; 
	\item  $\int_{0}^{t}{1_{\{X_{s}\neq0\}}dA_{s}}=0$ for all $t\geq0$.
\end{enumerate}
\end{defn}

Nikeghbali's Theorem 2.1 \cite{nik} serves as a method to characterize the processes that satisfy the assumptions of Definition \ref{dnik}. This result is called the characterization martingale theorem. We recall it as follows:

\begin{theorem}\label{tnik}
Let $X=M+A$ be a positive semi-martingale. Then, the following are equivalent:
\begin{enumerate}
  \item $X\in(\Sigma)$;
	\item There exists a non-decreasing continuous process $V$ with $V_{0}=0$ such that, for any locally bounded Borel function $f$ with $F(x)=\int_{0}^{x}{f(z)dz}$, the process 
	$$\left(F(V_{t})-f(V_{t})X_{t};t\geq0\right)$$
is a càdlàg local martingale and $V\equiv A$.
\end{enumerate}
\end{theorem}

This result was extended by Eyi Obiang et al. \cite{fjo} for càdlàg non-negative processes satisfying Definition \ref{d1}, as follows:

\begin{theorem}\label{teyi}
Let $X=M+A$ be a positive and càdlàg semi-martingale. Then, the following are equivalent:
\begin{enumerate}
  \item $X\in(\Sigma)$;
	\item There exists a càdlàg non-decreasing predictable process with finite variations $V$ such that, for any $f\in C^1$ with $F(x)=\int_{0}^{x}{f(z)dz}$, the process 
	$$\left(F(V^{c}_{t})-f(V^{c}_{t})X_{t}+\sum_{s\leq t}{[f(V^{c}_{s})-f^{'}(V^{c}_{s})X_{s}]\Delta V_{s}};t\geq0\right)$$
is a càdlàg local martingale and $V\equiv A$.
\end{enumerate}
\end{theorem}

 Now, recall that the family of processes of class $(\Sigma)$ considered by Nikeghbali (Definition \ref{dnik}) is also included in class $(\Sigma^{r})$. Indeed, it suffices to say that if $A$ is a continuous process, we have
$$\int_{0}^{t}{1_{\{X_{s-}\neq0\}}dA_{s}}=\int_{0}^{t}{1_{\{X_{s}\neq0\}}dA_{s}}=0.$$
Hence, in the next subsection, we present an extension of the characterization martingale theorem for the processes of class $(\Sigma^{r})$.

{\subsection{New characterization result for the class \texorpdfstring{$(\Sigma^{r})$}{sigma(r)}}}
 Let us begin with an extension of Lemma 2.3 of \cite{pat}.
\begin{lem}\label{lm1}
Let $X=M+A$ be a process of the class $(\Sigma^{r})$ and $A^{c}$ be the continuous part of $A$. For every $\mathcal{C}^{1}$ function $f$ and a function $F$ defined by $F(x)=\int_{0}^{x}{f(z)dz}$, the process 
$$\left(F(A^{c}_{t})-f(A^{c}_{t})X_{t}+\sum_{s\leq t}{[f(A^{c}_{s})-f^{'}(A^{c}_{s})X_{s-}]\Delta A_{s}};t\geq0\right)$$
is a càdlàg local martingale.
\end{lem}
\begin{proof}
Through integration by parts, we get
$$f(A^{c}_{t})X_{t}=\int_{0}^{t}{f(A^{c}_{s})dX_{s}}+\int_{0}^{t}{f^{'}(A^{c}_{s})X_{s-}dA^{c}_{s}}.$$
Hence, we have  
$$f(A^{c}_{t})X_{t}=\int_{0}^{t}{f(A^{c}_{s})dX_{s}}+\int_{0}^{t}{f^{'}(A^{c}_{s})X_{s-}dA_{s}}-\sum_{s\leq t}{f^{'}(A^{c}_{t})X_{s-}\Delta A_{s}}$$
because $A=A^{c}+\sum_{s\leq t}{\Delta A_{s}}$. Furthermore, we have $\int_{0}^{t}{f^{'}(A^{c}_{s})X_{s-}dA_{s}}=0$ since $dA$ is carried by $\{t\geq0:X_{t-}=0\}$. Therefore, it follows that
$$f(A^{c}_{t})X_{t}=\int_{0}^{t}{f(A^{c}_{s})dX_{s}}-\sum_{s\leq t}{f^{'}(A^{c}_{t})X_{s-}\Delta A_{s}}$$
$$\hspace{2cm}=\int_{0}^{t}{f(A^{c}_{s})dM_{s}}+\int_{0}^{t}{f(A^{c}_{s})dA^{c}_{s}}+\sum_{s\leq t}{[f(A^{c}_{s})-f^{'}(A^{c}_{s})X_{s-}]\Delta A_{s}}.$$
Consequently,
$$f(A^{c}_{t})X_{t}=\int_{0}^{t}{f(A^{c}_{s})dM_{s}}+F(A^{c}_{t})+\sum_{s\leq t}{[f(A^{c}_{s})-f^{'}(A^{c}_{s})X_{s-}]\Delta A_{s}}.$$
This implies that
$$F(A^{c}_{t})+\sum_{s\leq t}{[f(A^{c}_{s})-f^{'}(A^{c}_{s})X_{s-}]\Delta A_{s}}-f(A^{c}_{t})X_{t}=-\int_{0}^{t}{f(A^{c}_{s})dM_{s}}.$$
This completes the proof.
\end{proof}

Now, we shall present our martingale characterization theorem for the class $(\Sigma^{r})$.

\begin{theorem}\label{t1}
Let $X=M+A$ be a positive semi-martingale. Then, the following are equivalent:
\begin{enumerate}
  \item $X\in(\Sigma^{r})$;
	\item There exists a non-decreasing predictable process $V$ such that, for any $F\in C^2$, the process 
	$$\left(F(V^{c}_{t})-F^{'}(V^{c}_{t})X_{t}+\sum_{s\leq t}{[F^{'}(V^{c}_{s})-F^{''}(V^{c}_{s})X_{s-}]\Delta V_{s}};t\geq0\right)$$
is a càdlàg local martingale and $V\equiv A$.
\end{enumerate}
\end{theorem}
\begin{proof}
$(1)\Rightarrow(2)$ Let us consider $V=A$. Hence, from Lemma \ref{lm1}, we determine that
$$\left(F(A^{c}_{t})-F^{'}(A^{c}_{t})X_{t}+\sum_{s\leq t}{[F^{'}(A^{c}_{s})-F^{''}(A^{c}_{s})X_{s-}]\Delta A_{s}};t\geq0\right)$$
is a càdlàg local martingale.\\
$(2)\Rightarrow(1)$ First, let $F(x)=x$. Then, the process $W$ defined by 
$$W_{t}=V^{c}_{t}+\sum_{s\leq t}{\Delta V_{s}}-X_{t}=V_{t}-X_{t}$$
is a local martingale. Hence, owing to the uniqueness of the Doob--Meyer decomposition, we obtain $V=A$. Next, we take $F(x)=x^{2}$. Thus, process $B$ defined by 
$$B_{t}=(V_{t}^{c})^{2}-2V_{t}^{c}X_{t}+2\sum_{s\leq t}{V_{s}^{c}\Delta V_{s}}-2\sum_{s\leq t}{X_{s-}\Delta V_{s}}$$
is a local martingale. However, through integration by parts, it follows that
$$B_{t}=2\int_{0}^{t}{V^{c}_{s}dV^{c}_{s}}-2\int_{0}^{t}{V^{c}_{s}dX_{s}}-2\int_{0}^{t}{X_{s-}dV^{c}_{s}}+2\sum_{s\leq t}{V_{s}^{c}\Delta V_{s}}-2\sum_{s\leq t}{X_{s-}\Delta V_{s}}$$
$$\hspace{-0.75cm}=2\int_{0}^{t}{V^{c}_{s}d\left(V^{c}_{s}+\sum_{u\leq s}{\Delta V_{u}}-X_{s}\right)}-2\int_{0}^{t}{X_{s-}d\left(V^{c}_{s}+\sum_{u\leq s}{\Delta V_{u}}\right)}$$
$$\hspace{-6cm}=2\int_{0}^{t}{V^{c}_{s}dW_{s}}-2\int_{0}^{t}{X_{s-}dV_{s}}.$$
Consequently, we must have
$$\int_{0}^{t}{X_{s-}dV_{s}}=0.$$
In other words, $dA$ is carried  by the set $\{t\geq0:X_{t-}=0\}$.
\end{proof}

\section{Characterization of a new class of stochastic processes}\label{s2}

We propose unifying the study of the stochastic processes of classes $(\Sigma)$ and $(\Sigma^{r})$. More precisely, we provide a general framework to study a larger class that we term as class $(\Sigma^{g})$. 

{\subsection{First characterization and some properties}}

As is evident from the above definition, classes $(\Sigma)$ and $(\Sigma^{r})$ are included in class $(\Sigma^{g})$. Indeed, we can see that $\{X_{t}=0\}\subset\{X_{t}X_{t-}=0\}$ and $\{X_{t-}=0\}\subset\{X_{t}X_{t-}=0\}$. However, there exist processes of class $(\Sigma^{g})$ that do not belong to classes  $(\Sigma)$ and $(\Sigma^{r})$. Next, we present the first method to characterize the stochastic processes of class $(\Sigma^{g})$.
\begin{theorem}\label{t2}
Let $X=M+A$ be a càdlàg semi-martingale. Then, the following are equivalent:
\begin{enumerate}
	\item $X\in(\Sigma^{g})$;
	\item there exist two predictable processes $C$ and $V$ such that $A=C+V$ and 
	$$\int_{0}^{t}{1_{\{X_{s}\neq0\}}dC_{s}}=\int_{0}^{t}{1_{\{X_{s-}\neq0\}}dV_{s}}=0.$$
\end{enumerate}
\end{theorem}

\begin{proof}
$(1)\Rightarrow(2)$ We can see that, for all $t\geq0$,
$$A_{t}=\int_{0}^{t}{dA_{s}}=\int_{0}^{t}{1_{\{X_{s}=0\}}dA_{s}}+\int_{0}^{t}{1_{\{X_{s}\neq0\}}dA_{s}}$$
$$\hspace{3cm}=\int_{0}^{t}{1_{\{X_{s}=0\}}dA_{s}}+\int_{0}^{t}{1_{\{X_{s}X_{s-}\neq0\}}dA_{s}}+\int_{0}^{t}{1_{\{X_{s}\neq0=X_{s-}\}}dA_{s}}.$$
However,
$$\int_{0}^{t}{1_{\{X_{s}X_{s-}\neq0\}}dA_{s}}=0$$
as $dA_{s}$ is carried by $\{X_{s}X_{s-}=0\}$. Hence, it entails the following:
$$A_{t}=\int_{0}^{t}{1_{\{X_{s}=0\}}dA_{s}}+\int_{0}^{t}{1_{\{X_{s}\neq0=X_{s-}\}}dA_{s}}.$$
Now, let us substitute $C_{t}=\int_{0}^{t}{1_{\{X_{s}=0\}}dA_{s}}$ and $V_{t}=\int_{0}^{t}{1_{\{X_{s}\neq0=X_{s-}\}}dA_{s}}$. 
Thus, we obtain $\forall t\geq0$,
$$\int_{0}^{t}{1_{\{X_{s}\neq0\}}dC_{s}}=\int_{0}^{t}{1_{\{X_{s}\neq0\}}1_{\{X_{s}=0\}}dA_{s}}=0$$
and
$$\int_{0}^{t}{1_{\{X_{s-}\neq0\}}dV_{s}}=\int_{0}^{t}{1_{\{X_{s-}\neq0\}}1_{\{X_{s}\neq0=X_{s-}\}}dA_{s}}=0.$$
$(2)\Rightarrow(1)$\\
Now, assume that $A=C+V$ with $\int_{0}^{t}{1_{\{X_{s-}\neq0\}}dV_{s}}=\int_{0}^{t}{1_{\{X_{s}\neq0\}}dC_{s}}=0$. One has $\forall t\geq0$,
$$\int_{0}^{t}{1_{\{X_{s}X_{s-}\neq0\}}dA_{s}}=\int_{0}^{t}{1_{\{X_{s}X_{s-}\neq0\}}dC_{s}}+\int_{0}^{t}{1_{\{X_{s}X_{s-}\neq0\}}dV_{s}}.$$
However,
$$\int_{0}^{t}{1_{\{X_{s}X_{s-}\neq0\}}dC_{s}}=\int_{0}^{t}{1_{\{X_{s-}\neq0\}}1_{\{X_{s}\neq0\}}dC_{s}}=0\text{, since }1_{\{X_{s}\neq0\}}dC_{s}\equiv0$$
and
$$\int_{0}^{t}{1_{\{X_{s}X_{s-}\neq0\}}dV_{s}}=\int_{0}^{t}{1_{\{X_{s}\neq0\}}1_{\{X_{s-}\neq0\}}dV_{s}}=0\text{, since }1_{\{X_{s-}\neq0\}}dV_{s}\equiv0.$$
This completes the proof.
\end{proof}

As an application of Theorem \ref{t2}, we present two corollaries that provide a new approach to characterize the classes $(\Sigma)$ and $(\Sigma^{r})$.
\begin{coro}\label{c1}
Let $X=M+A$ be a càdlàg semi-martingale. Then, the following are equivalent:
\begin{enumerate}
	\item $X\in (\Sigma)$;
	\item there exist a continuous finite variation process $V$ and a càdlàg predictable process $C$ such that $A=C+V$ and 
	$$\int_{0}^{t}{1_{\{X_{s}\neq0\}}dC_{s}}=\int_{0}^{t}{1_{\{X_{s-}\neq0\}}dV_{s}}=0.$$
\end{enumerate}
\end{coro}
\begin{proof}
$(1)\Rightarrow(2)$ Assume that $X$ is an element of the class $(\Sigma)$. Hence, it follows from Definition \ref{d1} that there exists a local martingale $M$ and a càdlàg, predictable process $A$ such that $\forall t\geq0$, $dA_{t}$ is carried by $\{X_{t}=0\}$ and $X=M+A$. It is evident that (2) yields by taking $C=A$ and $V\equiv0$.\\
$(2)\Rightarrow(1)$ Now, assume that Assertion $(2)$ is true. We have $\forall t\geq0$,
	$$\int_{0}^{t}{1_{\{X_{s}\neq0\}}dA_{s}}=\int_{0}^{t}{1_{\{X_{s}\neq0\}}dC_{s}}+\int_{0}^{t}{1_{\{X_{s}\neq0\}}dV_{s}}.$$
	However, 
	$$\int_{0}^{t}{1_{\{X_{s}\neq0\}}dC_{s}}=0$$
	as $dC_{s}$ is carried by $\{X_{s}\neq0\}$. Hence,
	$$\int_{0}^{t}{1_{\{X_{s}\neq0\}}dA_{s}}=\int_{0}^{t}{1_{\{X_{s}\neq0\}}dV_{s}}.$$
	Furthermore,
	$$\int_{0}^{t}{1_{\{X_{s}\neq0\}}dV_{s}}=\int_{0}^{t}{1_{\{X_{s-}\neq0\}}dV_{s}}$$
	because $V$ is continuous. Therefor,
	$$\int_{0}^{t}{1_{\{X_{s}\neq0\}}dA_{s}}=\int_{0}^{t}{1_{\{X_{s-}\neq0\}}dV_{s}}=0.$$
	Consequently, $X$ is an element of the class $(\Sigma)$. This completes the proof.
\end{proof}

\begin{coro}\label{c2}
Let $X=M+A$ be a càdlàg stochastic process. Then, the following are equivalent:
\begin{enumerate}
	\item $X\in (\Sigma^{r})$;
	\item there exist a continuous finite variation process $C$ and a càdlàg predictable process $V$ such that $A=C+V$ and 
	$$\int_{0}^{t}{1_{\{X_{s}\neq0\}}dC_{s}}=\int_{0}^{t}{1_{\{X_{s-}\neq0\}}dV_{s}}=0.$$
\end{enumerate}
\end{coro}
\begin{proof}
$(1)\Rightarrow(2)$ Assume that $X$ is an element of class $(\Sigma^{r})$. Hence, there exist a local martingale $M$ and a càdlàg predictable process $A$ such that $\forall t\geq0$, $dA_{t}$ is carried by $\{X_{t-}=0\}$ and $X=M+A$. It is clear that (2) yields by taking $V=A$ and $C\equiv0$.\\
$(2)\Rightarrow(1)$ Now, assume that Assertion $(2)$ is true. We have $\forall t\geq0$,
	$$\int_{0}^{t}{1_{\{X_{s-}\neq0\}}dA_{s}}=\int_{0}^{t}{1_{\{X_{s-}\neq0\}}dC_{s}}+\int_{0}^{t}{1_{\{X_{s-}\neq0\}}dV_{s}}.$$
	However, 
	$$\int_{0}^{t}{1_{\{X_{s-}\neq0\}}dV_{s}}=0$$
	as $dV_{s}$ is carried by $\{X_{s-}\neq0\}$. Hence,
	$$\int_{0}^{t}{1_{\{X_{s-}\neq0\}}dA_{s}}=\int_{0}^{t}{1_{\{X_{s-}\neq0\}}dC_{s}}.$$
	Furthermore,
	$$\int_{0}^{t}{1_{\{X_{s-}\neq0\}}dC_{s}}=\int_{0}^{t}{1_{\{X_{s}\neq0\}}dC_{s}}$$
	because $C$ is continuous. Therefor,
	$$\int_{0}^{t}{1_{\{X_{s-}\neq0\}}dA_{s}}=\int_{0}^{t}{1_{\{X_{s}\neq0\}}dC_{s}}=0.$$
	Consequently, $X$ is an element of class $(\Sigma^{r})$. This completes the proof.

\end{proof}

Now, we explore some general properties of the stochastic processes of the class $(\Sigma^{g})$. Hence, we begin by deriving the properties using the balayage formulas:
\begin{lem}\label{l1}
Let $X$ be a process of class $(\Sigma^{g})$, and let $\gamma_{t}=\sup\{s\leq t:X_{s}=0\}$. Then, for any bounded predictable process $k$, $k_{\gamma_{\cdot}}X$ is also an element of class $(\Sigma^{g})$.
\end{lem}

\begin{proof}
By applying the balayage formula for the càdlàg case, we obtain the following:
$$k_{\gamma_{t}}X_{t}=k_{\gamma_{0}}X_{0}+\int_{0}^{t}{k_{\gamma_{s}}dX_{s}}=\int_{0}^{t}{k_{\gamma_{s}}dM_{s}}+\int_{0}^{t}{k_{\gamma_{s}}dC_{s}}+\int_{0}^{t}{k_{\gamma_{s}}dV_{s}}.$$
It is clear that $\int_{0}^{\cdot}{k_{\gamma_{s}}dM_{s}}$ is a local martingale; furthermore, $k_{\gt}dC_{t}$ is carried by $\{t\geq0:k_{\gamma_{t}}X_{t}=0\}$ and $k_{\gt}dV_{t}$ is carried by $\{t\geq0:k_{\gamma_{t-}}X_{t-}=0\}$. This completes the proof.
\end{proof}

\begin{coro}\label{c3}
Let $X=M+C+V=M+A$ be a process of class $(\Sigma^{g})$, and let $f$ be a bounded Borel function. Then, the process $\left(f(C_{t})X_{t}:t\geq0\right)$ is an element of the class $(\Sigma^{g})$, and its finite variation part is defined by $\forall t\geq0$, $A^{'}_{t}=\int_{0}^{t}{f(C_{s})d(C_{s}+V_{s})}$.
\end{coro}
\begin{proof}
According to Lemma \ref{l1}, $(f(C_{\gamma_{t}})X_{t}:t\geq0)$ is an element of the class $(\Sigma^{g})$. Furthermore, we have $\forall t\geq0$,
$$f(C_{\gamma_{t}})X_{t}=\int_{0}^{t}{f(C_{\gamma_{s}})dM_{s}}+\int_{0}^{t}{f(C_{\gamma_{s}})dA_{s}}.$$
As $dC_{t}$ is carried by $\{X_{t}=0\}$, we have $\forall t\geq0$, $C_{\gamma_{t}}= C_{t}$. Consequently, $\forall t\geq0$,
$$f(C_{t})X_{t}=\int_{0}^{t}{f(C_{s})dM_{s}}+\int_{0}^{t}{f(C_{s})dA_{s}}.$$
This completes the proof.
\end{proof}

\begin{coro}
Let $X=M+C+V$ be a positive process of the class $(\Sigma^{g})$. Then, there exist a càdlàg non-decreasing predictable process $\Gamma$ satisfying $Supp{(d\Gamma_{t})}\subset\{X_{t}=0\}$ and a positive submartingale $W=m+l$ with $W_{0}=1$; the measure $dl_{t}$ is carried by $\{X_{t-}=0\}$ such that $\forall t\geq0$,
$$X_{t}=\Gamma_{t}W_{t}-1.$$
\end{coro}
\begin{proof}
As the function $f$, defined by  $f(x)=e^{-x}$, is a bounded Borel function on $[0,+\infty[$, it follows from Corollary \ref{c3} that
$$f(C_{t})X_{t}-\int_{0}^{t}{f(C_{s})dC_{s}}=\int_{0}^{t}{f(C_{s})dM_{s}}+\int_{0}^{t}{f(C_{s})dV_{s}}.$$
Hence, we obtain that, $\forall t\geq0$,
$$e^{-C_{t}}(X_{t}+1)-1=\int_{0}^{t}{e^{-C_{s}}dM_{s}}+\int_{0}^{t}{e^{-C_{s}}dV_{s}}.$$
Therefore, considering $W_{t}=1+\int_{0}^{t}{e^{-C_{s}}dM_{s}}+\int_{0}^{t}{e^{-C_{s}}dV_{s}}$, we get
\begin{equation}\label{e1}
	e^{-C_{t}}(X_{t}+1)=W_{t}.
\end{equation}
Consequently,
$$X_{t}=\Gamma_{t}W_{t}-1,$$
where $\Gamma_{t}=e^{C_{t}}$. It is evident from \eqref{e1} that $W$ is a positive submartingale with $W_{0}=1$, and its non-decreasing part $l_{t}=\int_{0}^{t}{e^{-C_{s}}dV_{s}}$ is such that $Supp{(dl_{t})}\subset\{X_{t-}=0\}$.
\end{proof}

Now, we study the negative and positive parts of the stochastic processes of the class $(\Sigma^{g})$.
\begin{lem}\label{l3}
Let $X=M+A=M+C+V$ be a process of class $(\Sigma^{g})$. The following hold:
\begin{enumerate}
	\item If $C$ is a non-decreasing process, then $X^{+}$ is a local submartingale. 
	\item If $C$ is a decreasing process, then $X^{-}$ is a local submartingale.
	\item If $C$ has no negative jump and $\int_{0}^{t}{1_{\{X_{s}\neq0\}}dC^{c}_{s}}=0$, then $X^{+}$ is a local submartingale.
	\item If $C$ has no positive jump and $\int_{0}^{t}{1_{\{X_{s}\neq0\}}dC^{c}_{s}}=0$, then $X^{-}$ is a local submartingale.
\end{enumerate}
\end{lem}
\begin{proof}
 From Tanaka's formula, we have
	$$X_{t}^{+}=\int_{0}^{t}{1_{\{X_{s-}>0\}}dX_{s}}+\sum_{0<s\leq t}{1_{\{X_{s-}\leq0\}}X_{s}^{+}}+\sum_{0<s\leq t}{1_{\{X_{s-}>0\}}X_{s}^{-}}+\frac{1}{2}L_{t}^{0}.$$
	However,
	$$\int_{0}^{t}{1_{\{X_{s-}>0\}}dX_{s}}=\int_{0}^{t}{1_{\{X_{s-}>0\}}dM_{s}}+\int_{0}^{t}{1_{\{X_{s-}>0\}}dC_{s}}+\int_{0}^{t}{1_{\{X_{s-}>0\}}dV_{s}}.$$
	Hence,
	$$\int_{0}^{t}{1_{\{X_{s-}>0\}}dX_{s}}=\int_{0}^{t}{1_{\{X_{s-}>0\}}dM_{s}}+\int_{0}^{t}{1_{\{X_{s-}>0\}}dC_{s}}$$
	as $\int_{0}^{t}{1_{\{X_{s-}>0\}}dV_{s}}=0$; this is because $dV_{t}$ is carried by $\{X_{t-}=0\}$. Then,
	\begin{equation}\label{e+}
		X_{t}^{+}=\int_{0}^{t}{1_{\{X_{s-}>0\}}dM_{s}}+\int_{0}^{t}{1_{\{X_{s-}>0\}}dC_{s}}+\sum_{0<s\leq t}{1_{\{X_{s-}\leq0\}}X_{s}^{+}}+\sum_{0<s\leq t}{1_{\{X_{s-}>0\}}X_{s}^{-}}+\frac{1}{2}L_{t}^{0}.
	\end{equation}
	Thus, we have the following:
	\begin{enumerate}
		\item  We first remark that 
		$$\left(\int_{0}^{t}{1_{\{X_{s-}>0\}}dC_{s}}+\sum_{0<s\leq t}{1_{\{X_{s-}\leq0\}}X_{s}^{+}}+\sum_{0<s\leq t}{1_{\{X_{s-}>0\}}X_{s}^{-}}+\frac{1}{2}L_{t}^{0};t\geq0\right)$$
		is an increasing process that vanishes at zero, as $C$ is a non-decreasing process. 
		
		Furthermore, $M$ and $\int_{0}^{\cdot}{1_{\{X_{s-}>0\}}dM_{s}}$ are local martingales. Then, $X^{+}$ is a local submartingale. 
		\item Now, for any process of the class $(\Sigma^{g})$, $-X$ is again an element of the class $(\Sigma^{g})$.
		
		Therefore, it follows that $X^{-}=(-X)^{+}$ is a local submartingale when the process $C$ decreases.
		\item We obtain the following from identity \eqref{e+}:
		$$X_{t}^{+}=\int_{0}^{t}{1_{\{X_{s-}>0\}}dM_{s}}+\sum_{0<s\leq t}{1_{\{X_{s-}>0\}}\Delta C_{s}}+\sum_{0<s\leq t}{1_{\{X_{s-}\leq0\}}X_{s}^{+}}+\sum_{0<s\leq t}{1_{\{X_{s-}>0\}}X_{s}^{-}}+\frac{1}{2}L_{t}^{0},$$
		as 
		$$\int_{0}^{t}{1_{\{X_{s-}>0\}}dC_{s}}=\int_{0}^{t}{1_{\{X_{s-}>0\}}dC^{c}_{s}}+\sum_{0<s\leq t}{1_{\{X_{s-}>0\}}\Delta C_{s}}\text{ and }\int_{0}^{t}{1_{\{X_{s-}>0\}}dC^{c}_{s}}=\int_{0}^{t}{1_{\{X_{s}>0\}}dC^{c}_{s}}=0.$$
		Hence, 
		$$\left(\sum_{0<s\leq t}{1_{\{X_{s-}>0\}}\Delta C_{s}}+\sum_{0<s\leq t}{1_{\{X_{s-}\leq0\}}X_{s}^{+}}+\sum_{0<s\leq t}{1_{\{X_{s-}>0\}}X_{s}^{-}}+\frac{1}{2}L_{t}^{0};t\geq0\right)$$
		is an increasing process because $C$ has no negative jump. Consequently, $X^{+}$ is a local submartingale.
	\end{enumerate}
\end{proof}

\begin{rem}
A direct consequence is that any non-negative stochastic process of class $(\Sigma^{g})$ satisfying the assumptions of Lemma \ref{l3} is a submartingale.
\end{rem}
\begin{lem}\label{l4}
Let $X$ be a process of class $(\Sigma^{g})$. Hence, $X^{+}$ and $X^{-}$ are stochastic processes of class $(\Sigma^{g})$.
\end{lem}

\begin{proof}
Based on Tanaka's formula, we have
	$$X_{t}^{+}=\int_{0}^{t}{1_{\{X_{s-}>0\}}dX_{s}}+\sum_{0<s\leq t}{1_{\{X_{s-}\leq0\}}X_{s}^{+}}+\sum_{0<s\leq t}{1_{\{X_{s-}>0\}}X_{s}^{-}}+\frac{1}{2}L_{t}^{0}.$$
	However, 
	$$\int_{0}^{t}{1_{\{X_{s-}>0\}}dX_{s}}=\int_{0}^{t}{1_{\{X_{s-}>0\}}dM_{s}}+\int_{0}^{t}{1_{\{X_{s-}>0\}}dC_{s}}+\int_{0}^{t}{1_{\{X_{s-}>0\}}dV_{s}}=\int_{0}^{t}{1_{\{X_{s-}>0\}}dM_{s}}+\int_{0}^{t}{1_{\{X_{s-}>0\}}dC_{s}},$$
as $dV_{t}$ is carried by $\{X_{t-}=0\}$. Hence,
	\begin{equation}\label{*}
	X_{t}^{+}=\int_{0}^{t}{1_{\{X_{s-}>0\}}dM_{s}}+\int_{0}^{t}{1_{\{X_{s-}>0\}}dC_{s}}+\sum_{0<s\leq t}{1_{\{X_{s-}\leq0\}}X_{s}^{+}}+\sum_{0<s\leq t}{1_{\{X_{s-}>0\}}X_{s}^{-}}+\frac{1}{2}L_{t}^{0}.
	\end{equation}
Now, let us set $Y_{t}=\sum_{0<s\leq t}{1_{\{X_{s-}\leq0\}}X_{s}^{+}}$ and $Z_{t}=\sum_{0<s\leq t}{1_{\{X_{s-}>0\}}X_{s}^{-}}$. As $M$ and $\int_{0}^{\cdot}{1_{\{X_{s-}>0\}}dM_{s}}$ are local martingales and $C+V$ is a càdlàg, there exists a sequence of stopping times $(T_{n};n\in\N)$ increasing to $\infty$, such that 
	$$E[(X_{T_{n}})^{+}]=E[(M_{T_{n}}+C_{T_{n}}+V_{T_{n}})^{+}]<\infty\text{ and }E\left[\int_{0}^{T_{n}}{1_{\{X_{s-}>0\}}dM_{s}}\right]=0\text{, }n\in\N.$$
	It follows from Equation \eqref{*} that 
	$$E[Y_{T_{n}}]\leq E\left[(X_{T_{n}})^{+}-\int_{0}^{t}{1_{\{X_{s-}>0\}}dC_{s}}\right]<\infty$$
	and 
	$$E[Z_{T_{n}}]\leq E\left[(X_{T_{n}})^{+}-\int_{0}^{t}{1_{\{X_{s-}>0\}}dC_{s}}\right]<\infty$$
	for all $n\in\N$. Thus, based on Theorem VI.80 of \cite{pot}, there exist right continuous increasing predictable processes $V^{Y}$ and $V^{Z}$ such that $Y-V^{Y}$ and $Z-V^{Z}$ are local martingales vanishing at zero. Moreover, there exists a sequence of stopping times $(R_{n};n\in\N)$ increasing to $\infty$, such that
	$$E\left[\int_{0}^{t\wedge R_{n}}{1_{\{X^{+}_{s-}\neq0\}}dV_{s}^{Y}}\right]=E\left[\int_{0}^{t\wedge R_{n}}{1_{\{X^{+}_{s-}\neq0\}}d(V_{s}^{Y}-Y_{s})}+\int_{0}^{t\wedge R_{n}}{1_{\{X^{+}_{s-}\neq0\}}dY_{s}}\right].$$
	As $\int_{0}^{\cdot\wedge R_{n}}{1_{\{X^{+}_{s-}\neq0\}}d(V_{s}^{Y}-Y_{s})}$ is a local martingale, it entails that
	$$E\left[\int_{0}^{t\wedge R_{n}}{1_{\{X^{+}_{s-}\neq0\}}dV_{s}^{Y}}\right]=E\left[\int_{0}^{t\wedge R_{n}}{1_{\{X^{+}_{s-}\neq0\}}dY_{s}}\right].$$
	Therefore,
	$$E\left[\int_{0}^{t\wedge R_{n}}{1_{\{X^{+}_{s-}\neq0\}}dV_{s}^{Y}}\right]=E\left[\sum_{0< s\leq t\wedge R_{n}}{1_{\{X^{+}_{s-}\neq0\}}1_{\{X_{s-}\leq0\}}X_{s}^{+}}\right]=E\left[\sum_{0< s\leq t\wedge R_{n}}{1_{\{X^{+}_{s-}\neq0\}}1_{\{X^{+}_{s-}=0\}}X_{s}^{+}}\right]=0.$$
	In other words, $\int_{0}^{t}{1_{\{X^{+}_{s-}\neq0\}}dV_{s}^{Y}}=0$. Then, $dV_{t}^{Y}$ is carried by $\{X^{+}_{t-}=0\}$. However, we have
	$$E\left[\int_{0}^{t\wedge R_{n}}{1_{\{X^{+}_{s}\neq0\}}dV_{s}^{Z}}\right]=E\left[\int_{0}^{t\wedge R_{n}}{1_{\{X^{+}_{s}\neq0\}}d(V_{s}^{Z}-Z_{s})}+\int_{0}^{t\wedge R_{n}}{1_{\{X^{+}_{s}\neq0\}}dB_{s}}\right]$$
	$$\hspace{-0.5cm}=E\left[\int_{0}^{t\wedge R_{n}}{1_{\{X^{+}_{s}\neq0\}}adz_{s}}\right].$$
	Hence,
	$$E\left[\int_{0}^{t\wedge R_{n}}{1_{\{X^{+}_{s}\neq0\}}dV_{s}^{Z}}\right]=E\left[\sum_{0<s\leq t\wedge R_{n}}{1_{\{X_{s}^{+}\neq0\}}1_{\{X_{s-}>0\}}X_{s}^{-}}\right]=E\left[\sum_{0<s\leq t\wedge R_{n}}{1_{\{X_{s}>0\}}1_{\{X_{s-}>0\}}X_{s}^{-}}\right].$$
	This entails that
	$$E\left[\int_{0}^{t\wedge R_{n}}{1_{\{X^{+}_{s}\neq0\}}dV_{s}^{Z}}\right]=0,$$
	since $1_{\{X_{s}>0\}}X_{s}^{-}=0$. This shows that $\int_{0}^{t}{1_{\{X_{s}^{+}\neq0\}}dV_{s}^{Z}}=0$. Therefore, $dV_{t}^{Z}$ is carried by $\{t\geq0;X_{t}^{+}=0\}$. Consequently, we determine that
	$$X^{+}_{t}=\left(\int_{0}^{t}{1_{\{X_{s-}>0\}}dM_{s}}+(Y_{t}-V_{t}^{Y})+(Z_{t}-V_{t}^{Z})\right)+V_{t}^{Y}+\left(V_{t}^{Z}+\int_{0}^{t}{1_{\{X_{s-}>0\}}dC_{s}}+\frac{1}{2}L_{t}^{0}\right)$$
	is a stochastic process of the class $(\Sigma^{g})$. This is also true for $X^{-}$ as $(-X)$ is also from class $(\Sigma^{g})$.
\end{proof}

It is well known that $M^{+}$ and $M^{-}$ are stochastic processes of class $(\Sigma)$ when $M$ is a continuous local martingale. The next corollary of Lemma \ref{l4} shows that $M^{+}$ and $M^{-}$ are elements of class $(\Sigma^{g})$ when $M$ is a càdlàg local martingale.

\begin{coro}
Let $M$ be a càdlàg local martingale vanishing at zero. Then, the processes $M^{+}$ and $M^{-}$ are elements of class $(\Sigma^{g})$.
\end{coro}

Now, we show that the product of the processes of class $(\Sigma^{g})$ with vanishing quadratic covariations is again of class $(\Sigma^{g})$.
\begin{lem}\label{l5}
Let $(X_{t}^{1})_{t\geq0},\cdots,(X_{t}^{n})_{t\geq0}$ be processes of class $(\Sigma^{g})$, such that $\langle X^{i},X^{j}\rangle=0$ for $i\neq j$. Then, $(\Pi_{i=1}^{n}{X^{i}_{t}})_{t\geq0}$ is also of class $(\Sigma^{g})$.
\end{lem}
\begin{proof}
As $\langle X^{1},X^{2}\rangle=0$, integration by parts yields
$$X^{1}_{t}X^{2}_{t}=\int_{0}^{t}{X_{s-}^{1}dX_{s}^{2}}+\int_{0}^{t}{X_{s-}^{2}dX_{s}^{1}}.$$
In other words,
$$X^{1}_{t}X^{2}_{t}=\left[\int_{0}^{t}{X_{s-}^{1}dM_{s}^{2}}+\int_{0}^{t}{X_{s-}^{2}dM_{s}^{1}}\right]+\left[\int_{0}^{t}{X_{s-}^{1}dC_{s}^{2}}+\int_{0}^{t}{X_{s-}^{2}dC_{s}^{1}}\right]+\left[\int_{0}^{t}{X_{s-}^{1}dV_{s}^{2}}+\int_{0}^{t}{X_{s-}^{2}dV_{s}^{1}}\right].$$
It can be observed that $M_{t}=\int_{0}^{t}{X_{s-}^{1}dM_{s}^{2}}+\int_{0}^{t}{X_{s-}^{2}dM_{s}^{1}}$ is a càdlàg local martingale. Furthermore, the process $C_{t}=\int_{0}^{t}{X_{s-}^{1}dC_{s}^{2}}+\int_{0}^{t}{X_{s-}^{2}dC_{s}^{1}}$ is a finite variation process, such that
$$dC_{t}=X_{t-}^{1}dC_{t}^{2}+X_{t-}^{2}dC_{t}^{1}$$
is carried by $\{t\geq0:X^{1}_{t}X^{2}_{t}=0\}$. By contrast, $V_{t}=\int_{0}^{t}{X_{s-}^{1}dV_{s}^{2}}+\int_{0}^{t}{X_{s-}^{2}dV_{s}^{1}}$ is a finite variation process, such that
$$dV_{t}=X_{t-}^{1}dV_{t}^{2}+X_{t-}^{2}dV_{t}^{1}$$
is carried by $\{t\geq0:X^{1}_{t-}X^{2}_{t-}=0\}$. Therefore, $X^{1}X^{2}$ is of class $(\Sigma^{g})$. If $n\geq3$, and $\langle X^{1}X^{2},X^{3}\rangle=0$. Thus, we obtain the result by induction.
\end{proof}

\begin{theorem}\label{t3}
Let $X=M+C+V$ be a process of class $(\Sigma^{g} D)$. Then, there exists a random variable $X_{\infty}$ such that $$\lim_{t\to+\infty}{X_{t}}=X_{\infty}$$, and for every stopping time $T<\infty$, we have
\begin{equation}
	X_{T}=E\left[X_{\infty}1_{\{g<T\}}|\mathcal{F}_{T}\right],
\end{equation}
where $g=\sup{\{t\geq0:X_{t}X_{t-}=0\}}$.
\end{theorem}
\begin{proof}
Let us substitute $\gamma_{t}=\inf{\{s>t\geq0:X_{s}X_{s-}=0\}}$. It is evident that $\gamma_{t}$ is the stopping time. Furthermore, 
$$X_{\infty}1_{\{g<T\}}=X_{\gamma_{T}}=M_{\gamma_{T}}+C_{\gamma_{T}}+V_{\gamma_{T}}.$$
However, $C_{\gamma_{T}}=C_{T}$ and $V_{\gamma_{T}}=V_{T}$ as $dC$ and $dV$ are carried by $\{t\geq0:X_{t}=0\}$ and $\{t\geq0:X_{t-}=0\}$, respectively; further, $g=\sup{\{t\geq0:X_{t}=0\}}\vee\sup{\{t\geq0:X_{t-}=0\}}$. This entails that 
$$X_{\infty}1_{\{g<T\}}=X_{\gamma_{T}}=M_{\gamma_{T}}+C_{T}+V_{T}.$$
Hence,
$$E\left[X_{\infty}1_{\{g<T\}}|\mathcal{F}_{T}\right]=E\left[M_{\gamma_{T}}|\mathcal{F}_{T}\right]+C_{T}+V_{T}.$$
Therefore,
$$X_{T}=E\left[X_{\infty}1_{\{g<T\}}|\mathcal{F}_{T}\right]$$
as $M$ is a uniformly integrable martingale.
\end{proof}

\begin{coro}
Let $M$ be a non-negative càdlàg uniformly integrable martingale such that $M_{0}>0$ and $\lim_{t\to+\infty}{M_{t}}=0$. Let us consider $k>0$. Then, 
$$P\left(g_{k}\geq t|\mathcal{F}_{t}\right)=1\wedge\left(\frac{M_{t}}{k}\right),$$
where $g_{k}=\sup{\{t\geq0:M_{t}\geq k\text{ or }M_{t-}\geq k\}}$
\end{coro}
\begin{proof}
It follows from Theorem \ref{t1} that
$$(k-M_{t})^{+}=E\left[k1_{\{g_{k}<t\}}|\mathcal{F}_{t}\right]=kE\left[1_{\{g_{k}<t\}}|\mathcal{F}_{t}\right].$$
Hence,
$$(k-M_{t})^{+}=kP\left(g_{k}<t|\mathcal{F}_{t}\right).$$
Consequently, 
$$P\left(g_{k}<t|\mathcal{F}_{t}\right)=\left(1-\frac{M_{t}}{k}\right)^{+}.$$
Therefore,
$$P\left(g_{k}\geq t|\mathcal{F}_{t}\right)=1-\left(1-\frac{M_{t}}{k}\right)^{+}=1\wedge\left(\frac{M_{t}}{k}\right).$$
This completes the proof.
\end{proof}

{\subsection{Extension of characterization martingale}}

\begin{lem}\label{lm6}
Let $X=M+A$ be a process of class $(\Sigma^{g})$, where $A=C+V$ and $A^{c}$ denote the continuous part of $A$. Then, for every $\mathcal{C}^{1}$ function $f$ and \\$F(x)=\int_{0}^{x}{f(z)dz}$, the process 
$$\left(F(A^{c}_{t})-f(A^{c}_{t})X_{t}+\sum_{s\leq t}{[f(A^{c}_{s})-f^{'}(A^{c}_{s})X_{s}]\Delta C_{s}}+\sum_{s\leq t}{[f(A^{c}_{s})-f^{'}(A^{c}_{s})X_{s-}]\Delta V_{s}};t\geq0\right)$$
is a local martingale.
\end{lem}
\begin{proof}
Integration by parts yields
$$f(A^{c}_{t})X_{t}=\int_{0}^{t}{f(A^{c}_{s})dX_{s}}+\int_{0}^{t}{f^{'}(A^{c}_{s})X_{s-}dA^{c}_{s}}.$$
In other words, 
$$f(A^{c}_{t})X_{t}=\int_{0}^{t}{f(A^{c}_{s})dX_{s}}+\int_{0}^{t}{f^{'}(A^{c}_{s})X_{s-}dC^{c}_{s}}+\int_{0}^{t}{f^{'}(A^{c}_{s})X_{s-}dV^{c}_{s}}$$
since $A^{c}=C^{c}+V^{c}$. Furthermore, 
$$\int_{0}^{t}{f^{'}(A^{c}_{s})X_{s-}dC^{c}_{s}}=\int_{0}^{t}{f^{'}(A^{c}_{s})X_{s}dC^{c}_{s}}$$
because $C^{c}$ is continuous. Therefore, we obtain
$$f(A^{c}_{t})X_{t}=\int_{0}^{t}{f(A^{c}_{s})dX_{s}}+\int_{0}^{t}{f^{'}(A^{c}_{s})X_{s}dC^{c}_{s}}+\int_{0}^{t}{f^{'}(A^{c}_{s})X_{s-}dV^{c}_{s}}.$$
This entails the following:
$$f(A^{c}_{t})X_{t}=\int_{0}^{t}{f(A^{c}_{s})dX_{s}}+\left[\int_{0}^{t}{f^{'}(A^{c}_{s})X_{s}dC_{s}}-\sum_{s\leq t}{f^{'}(A^{c}_{s})X_{s}\Delta C_{s}}\right]+\left[\int_{0}^{t}{f^{'}(A^{c}_{s})X_{s-}dV_{s}}-\sum_{s\leq t}{f^{'}(A^{c}_{s})X_{s-}\Delta V_{s}}\right].$$
Thus, it follows that
$$f(A^{c}_{t})X_{t}=\int_{0}^{t}{f(A^{c}_{s})dX_{s}}-\sum_{s\leq t}{f^{'}(A^{c}_{s})X_{s}\Delta C_{s}}-\sum_{s\leq t}{f^{'}(A^{c}_{s})X_{s-}\Delta V_{s}}$$
as $dC$ and $dV$ are carried by $\{t\geq0:X_{t}=0\}$ and $\{t\geq0:X_{t-}=0\}$, respectively. This entails that
$$f(A^{c}_{t})X_{t}=\int_{0}^{t}{f(A^{c}_{s})dM_{s}}+\int_{0}^{t}{f(A^{c}_{s})dA^{c}_{s}}+\sum_{s\leq t}{\left[f(A^{c}_{s})-f^{'}(A^{c}_{s})X_{s}\right]\Delta C_{s}}+\sum_{s\leq t}{\left[f(A^{c}_{s})-f^{'}(A^{c}_{s})X_{s-}\right]\Delta V_{s}}.$$
Consequently,
$$F(A^{c}_{t})-f(A^{c}_{t})X_{t}+\sum_{s\leq t}{\left[f(A^{c}_{s})-f^{'}(A^{c}_{s})X_{s}\right]\Delta C_{s}}+\sum_{s\leq t}{\left[f(A^{c}_{s})-f^{'}(A^{c}_{s})X_{s-}\right]\Delta V_{s}}=-\int_{0}^{t}{f(A^{c}_{s})dM_{s}}.$$
In other words,
$$\left(F(A^{c}_{t})-f(A^{c}_{t})X_{t}+\sum_{s\leq t}{[f(A^{c}_{s})-f^{'}(A^{c}_{s})X_{s}]\Delta C_{s}}+\sum_{s\leq t}{[f(A^{c}_{s})-f^{'}(A^{c}_{s})X_{s-}]\Delta V_{s}};t\geq0\right)$$
is a local martingale.
\end{proof}

\begin{theorem}\label{t5}
Let $X=M+A$ be a positive semi-martingale. Then, the following are equivalent:
\begin{enumerate}
  \item $X\in(\Sigma^{g})$;
	\item There exists two càdlàg and non-decreasing predictable processes $V$ and $C$ such that, for $W=C+V$ and for any $F\in C^2$, the process 
	$$\left(F(W^{c}_{t})-F^{'}(W^{c}_{t})X_{t}+\sum_{s\leq t}{[F^{'}(W^{c}_{s})-F^{''}(W^{c}_{s})X_{s}]\Delta C_{s}}+\sum_{s\leq t}{[F^{'}(W^{c}_{s})-F^{''}(W^{c}_{s})X_{s-}]\Delta V_{s}};t\geq0\right)$$
is a càdlàg local martingale and $W\equiv A$.
\end{enumerate}
\end{theorem}
\begin{proof}
$(1)\Rightarrow(2)$ Let us consider $W=A$. Hence, from Lemma \ref{lm6}, we obtain
$$\left(F(A^{c}_{t})-F^{'}(A^{c}_{t})X_{t}+\sum_{s\leq t}{[F^{'}(A^{c}_{s})-F^{''}(A^{c}_{s})X_{s}]\Delta C_{s}}+\sum_{s\leq t}{[F^{'}(A^{c}_{s})-F^{''}(A^{c}_{s})X_{s-}]\Delta V_{s}};t\geq0\right)$$
is a càdlàg local martingale.\\
$(2)\Rightarrow(1)$ First, let $F(x)=x$. Then, the process $W^{'}$ defined by 
$$W^{'}_{t}=W^{c}_{t}+\sum_{s\leq t}{\Delta C_{s}}+\sum_{s\leq t}{\Delta V_{s}}-X_{t}=W_{t}-X_{t}$$
is a local martingale. Hence, owing to the uniqueness of the Doob--Meyer decomposition, we obtain $W=A$. Next, we consider $F(x)=x^{2}$. Then, the process $B$ defined by 
$$B_{t}=(W_{t}^{c})^{2}-2W_{t}^{c}X_{t}+2\sum_{s\leq t}{W_{s}^{c}\Delta C_{s}}+2\sum_{s\leq t}{W_{s}^{c}\Delta V_{s}}-2\sum_{s\leq t}{X_{s}\Delta C_{s}}-2\sum_{s\leq t}{X_{s-}\Delta V_{s}}$$
is a local martingale. However, through integration by part, it follows that
$$\hspace{-1.5cm}B_{t}=2\int_{0}^{t}{W^{c}_{s}dW^{c}_{s}}-2\int_{0}^{t}{W^{c}_{s}dX_{s}}-2\int_{0}^{t}{X_{s}dW^{c}_{s}}+2\sum_{s\leq t}{W_{s}^{c}\Delta C_{s}}+2\sum_{s\leq t}{W_{s}^{c}\Delta V_{s}}-2\sum_{s\leq t}{X_{s}\Delta C_{s}}-2\sum_{s\leq t}{X_{s-}\Delta V_{s}}$$
$$\hspace{-0.75cm}=2\int_{0}^{t}{W^{c}_{s}dW^{c}_{s}}-2\int_{0}^{t}{W^{c}_{s}dX_{s}}-2\int_{0}^{t}{X_{s}dC^{c}_{s}}-2\int_{0}^{t}{X_{s}dV^{c}_{s}}+2\sum_{s\leq t}{W_{s}^{c}\Delta C_{s}}+2\sum_{s\leq t}{W_{s}^{c}\Delta V_{s}}-2\sum_{s\leq t}{X_{s}\Delta C_{s}}-2\sum_{s\leq t}{X_{s-}\Delta V_{s}}$$
$$\hspace{-0.75cm}=2\int_{0}^{t}{W^{c}_{s}dW^{c}_{s}}-2\int_{0}^{t}{W^{c}_{s}dX_{s}}-2\int_{0}^{t}{X_{s}dC^{c}_{s}}-2\int_{0}^{t}{X_{s-}dV^{c}_{s}}+2\sum_{s\leq t}{W_{s}^{c}\Delta C_{s}}+2\sum_{s\leq t}{W_{s}^{c}\Delta V_{s}}-2\sum_{s\leq t}{X_{s}\Delta C_{s}}-2\sum_{s\leq t}{X_{s-}\Delta V_{s}}$$
$$\hspace{-0.75cm}=2\int_{0}^{t}{W^{c}_{s}d\left(W^{c}_{s}+\sum_{u\leq s}{\Delta C_{u}}+\sum_{u\leq s}{\Delta V_{u}}-X_{s}\right)}-2\int_{0}^{t}{X_{s}d\left(C^{c}_{s}+\sum_{u\leq s}{\Delta C_{u}}\right)}-2\int_{0}^{t}{X_{s-}d\left(V^{c}_{s}+\sum_{u\leq s}{\Delta V_{u}}\right)}$$
$$\hspace{-10.5cm}=2\int_{0}^{t}{W^{c}_{s}dW^{'}_{s}}-2\int_{0}^{t}{X_{s}dC_{s}}-2\int_{0}^{t}{X_{s-}dV_{s}}.$$
Consequently, we must have
$$\int_{0}^{t}{X_{s}dC_{s}}+\int_{0}^{t}{X_{s-}dV_{s}}=0.$$
Hence, we determine that
$$\int_{0}^{t}{X_{s}dC_{s}}=\int_{0}^{t}{X_{s-}dV_{s}}=0,$$
as $\int_{0}^{t}{X_{s}dC_{s}}$ and $\int_{0}^{t}{X_{s-}dV_{s}}$ are non-negative. In other words, $dA$ is carried by the set $\{t\geq0:X_{t}X_{t-}=0\}$.
\end{proof}

\begin{coro}
Let $X=M+A$ be a positive semi-martingale. Then, the following are equivalent:
\begin{enumerate}
  \item $X\in(\Sigma)$;
	\item There exists a non-decreasing predictable process $V$ such that, for any $F\in C^2$, the process 
	$$\left(F(V^{c}_{t})-F^{'}(V^{c}_{t})X_{t}+\sum_{s\leq t}{[F^{'}(V^{c}_{s})-F^{''}(V^{c}_{s})X_{s}]\Delta V_{s}};t\geq0\right)$$
is a càdlàg local martingale and $V\equiv A$.
\end{enumerate}
\end{coro}

\begin{coro}
Let $X=M+A$ be a positive semi-martingale. Then, the following are equivalent:
\begin{enumerate}
  \item $X\in(\Sigma^{r})$;
	\item There exists a non-decreasing predictable process $V$ such that, for any $F\in C^2$, the process 
	$$\left(F(V^{c}_{t})-F^{'}(V^{c}_{t})X_{t}+\sum_{s\leq t}{[F^{'}(V^{c}_{s})-F^{''}(V^{c}_{s})X_{s-}]\Delta V_{s}};t\geq0\right)$$
is a càdlàg local martingale and $V\equiv A$.
\end{enumerate}
\end{coro}

\begin{coro}
Let $X=M+A$ be a positive semi-martingale. Then, the following are equivalent:
\begin{enumerate}
  \item $X\in(\Sigma)$ (in sense of Nikeghbali\cite{nik} );
	\item There exists a non-decreasing predictable process $V$ such that for any $F\in C^2$, the process 
	$$\left(F(V_{t})-F^{'}(V_{t})X_{t};t\geq0\right)$$
is a càdlàg local martingale and $V\equiv A$.
\end{enumerate}
\end{coro}

\section{Conclusion}\label{s3}

The objective of this paper was to provide a new framework for studying the extensions of class $(\Sigma)$ when the finite variational part is considered to be càdlàg instead of continuous. More precisely, the objective was to contribute to the study of processes of the form
$$X=M+A,$$
where $M$ is a càdlàg martingale with $M_{0}=0$ and $A$ is a càdlàg predictable process of finite variation with $A_{0}=0$, such that the signed measure induced by $A$ is carried by one of the optional random sets $\{t\geq0:X_{t}=0\}$ and $\{t\geq0:X_{t-}=0\}$. First, we developed new approaches to characterize such stochastic processes. Then, we provided a general framework unifying the study of the two above-mentioned classes by presenting a new larger class. 

\section*{Declaration of competing interest}
This work does not have any conflicts of interest.

\section*{Funding}
This research did not receive any specific grant from funding agencies in the public, commercial, or not-for-profit sectors.

{\color{myaqua}

\end{document}